\documentclass{amsart}
\usepackage[english]{babel}
\usepackage[cp1250]{inputenc}
\usepackage{amssymb,amsthm}

\def\NN{\mathbb{N}}

\def\RR{\mathbb{R}}
\def\CC{\mathbb{C}}
\def\SSS{\mathbb{S}}
\def\DD{\mathbb{D}}
\def\Im{\mathrm{Im}}

\def\Lin{\mathrm{Lin}}

\newtheorem{proposition}{Proposition}

\newtheorem{theorem}{Theorem}
\newtheorem{lemma}{Lemma}
\newtheorem{example}{Example}

\begin{document}
\title{A note on a matrix version of the Farkas Lemma}

\author{Alja\v z Zalar}

\address{Alja\v z Zalar, University of Ljubljana, Faculty of  Mathematics and Physics, Slovenia}
\email{aljaz.zalar@student.fmf.uni-lj.si}

\date{\today}

\begin{abstract} A linear polyomial non-negative on the non-negativity domain of 
finitely many linear polynomials can be expressed as their non-negative linear combination.
Recently, under several additional assumptions, Helton, Klep, and McCullough extended this result to matrix polynomials.
The aim of this paper is to study which of these additional assumptions are really necessary.
\end{abstract}



\maketitle

\section{Introduction}

We are interested in matrix generalizations of the following variant of the Farkas lemma.

\begin{theorem} \label{farkas}
Let $f_1,f_2,\ldots,f_k$ be linear polynomials in $n$ variables, i.e., 
$f_i(x_1,x_2,$ $\ldots,x_n)=a^{(i)}_0+a^{(i)}_1 x_1+\ldots+a^{(i)}_n x_n$, where $a_j^{(i)}\in \RR$ for $i=1,2, \ldots,k$, $j=0,1, \ldots,n$. Let $K=\left\{x\in \RR^n \mid f_i(x)\geq 0 \; \mathrm{for}\;\mathrm{all}\; i=1,2, \ldots,k \right\}$. If $f$ is another linear polynomial in $n$ variables, for which $f|_K\geq 0$ holds, then there exist non-negative constants $c_i$, such that 
$$f=c_0+c_1 f_1+ c_2 f_2 + \ldots + c_k f_k.$$ 
\end{theorem}

We write $\RR^{d\times d}[x]$ (resp. $\SSS\RR^{d\times d}[x]$) for the set of all polynomials whose coefficients are 
$d \times d$ (resp. symmetric $d\times d$) matrices. The evaluation of a linear polynomial $L(x)=P_0+\sum_{i=1}^{n}P_ix_i\in \SSS\RR^{d\times d}[x]$ at $X=(X_1,X_2,\ldots,X_n)\in (\SSS\RR^{m\times m})^{n}$ is defined as
$$L(X)=P_0\otimes I_m+\sum_{i=1}^{n}P_i\otimes X_i\in \SSS\RR^{dm\times dm}$$
with the usual tensor product of matrices.
Let
\begin{eqnarray*}
D_{L}(m)&=&\{X \in (\SSS\RR^{m\times m})^n \mid L(X) \succeq 0\},\\
D_L&=&\bigcup_{m=1}^\infty D_L(m).
\end{eqnarray*}
The following generalization of Theorem \ref{farkas} was obtained by Helton, Klep, and McCullough in \cite{main}. It is a special case of their Theorem 6.1.

\begin{theorem} \label{main}
Suppose $L_1=I+\sum_{i=1}^{n}P_ix_i\in \SSS\RR^{d\times d}[x]$ is a monic linear polynomial and the set
$D_{L_1}(1)= \{x \in \RR^n \mid L_1(x) \succeq 0\}$ is bounded.  Then
for every linear polynomial $L_2=R_0+\sum_{i=1}^{n}R_ix_i\in \SSS\RR^{\ell\times \ell}[x]$ with $L_2|_{D_{L_1}(1)} \succ 0$, there are matrix polynomials $A_j\in \RR^{\ell\times \ell}\left[ x\right]$ and $B_k\in \RR^{d\times \ell}\left[x\right]$ satisfying 
$$L_2=\sum_jA_j^{\ast}A_j+\sum_kB_k^{\ast}L_1 B_k.$$
\end{theorem}

Note that this result also covers the case of several constraints; simply take $L_1$ to be their direct sum.

It was already shown in \cite{main} that the matrix polynomials $A_j, B_k$ need not be constant unless the condition $L_2|_{D_{L_{1}}(1)}\succ 0$ is replaced with $L_2|_{D_{L_{1}}}\succeq 0$; cf. [2, LP-satz].

The aim of this paper is to study the necessity of the following assumptions in Theorem \ref{main}:

\begin{enumerate}
\item \textbf{Boundedness of $D_{L_1}(1)$:} By Example $1$ this assumption cannot be removed.
\item \textbf{Monicity of $L_1$:} In Theorem \ref{diagonal} we prove that for diagonal $L_1$ monicity can be removed from Theorem $\ref{main}$. 
The general case remains open.
\item \textbf{Strict positivity of $L_2|_{D_{L_1}(1)}$:} In Section $4$ we show that this assumption can be replaced with 
$L_2|_{D_{L_1}(1)} \succeq 0$ in the following special cases:
\begin{itemize}
\item In the one-variable case (even if $D_{L_1}(1)$ is unbounded).
\item When the span of the coefficients of $L_1$ is closed under multiplication.
\end{itemize}
The general case remains open. However, by Example 2 simultaneous generalization to non-monic $L_1$
and non-strict $L_2|_{D_{L_1}(1)}$ is not possible.
\end{enumerate}

\section{Boundedness of $D_{L_1}(1)$}

The assumption of boundedness in Theorem \ref{main} cannot be removed, because of the following example: 

\begin{example}
For linear polynomials
$$
L_1=
\left[
\begin{array}{ccc}
1+x_1&0&0\\
0&1+x_1+x_2&0\\
0&0&1+x_2\\
\end{array}
\right], \quad 
L_2=\left[
\begin{array}{cc}
1+\frac{1}{3}x_1&\frac{3}{4}\\
\frac{3}{4}&1+\frac{1}{3}x_2\\
\end{array}
\right]$$
we have that the set $D_{L_1}(1)$ is contained in the set $\tilde{D}_{L_2}(1):=\left\{(x_1, x_2)\in \RR^2\mid \right.$ $\left.L_2(x_1,x_2)\succ 0\right\}$, but $L_2$ cannot be expressed as $\sum_j A_j^{\ast}A_j+\sum_k B_k^{\ast}L_1 B_k$. This implies that the boundedness of $D_{L_1}(1)$ is needed in Theorem $2$.
\end{example}

\begin{proof}
We have $D_{L_1}(1)=\left\{(x_1,x_2)\in \RR^2\mid x_1\geq -1, x_1+x_2\geq -1, x_2\geq -1\right\}$.
On the other hand, the fact $1+\frac{1}{3}x_1>0$ on $D_{L_1}(1)$ together with
\begin{eqnarray*}
\det(L_2) &=&
\left(1+\frac{1}{3}x_1\right)\left(1+\frac{1}{3}x_2\right)-\left(\frac{3}{4}\right)^2=\\
&=&\left(\frac{2}{3}+\frac{1}{3}(1+x_1)\right)\left(\frac{2}{3}+\frac{1}{3}(1+x_2)\right)-\left(\frac{3}{4}\right)^2=\\
&=&\left(\frac{2}{3}\right)^2+\frac{2}{9}\left(1+x_1+1+x_2\right)+\frac{1}{9}\left(1+x_1\right) \left(1+x_2\right)-\left(\frac{3}{4}\right)^2=\\
&=&\frac{5}{48}+\frac{2}{9}\left(1+x_1+x_2\right)+\frac{1}{9}\left(1+x_1\right)\left(1+x_2\right),
\end{eqnarray*}
where the second and the third summand in the last line are non-negative on $D_{L_1}(1)$, gives $L_2|_{D_{L_1}(1)}\succ 0$.

Now we are going to show, that if $L_2$ can be expressed as $L_2=\sum_j A_j^{\ast}A_j+\sum_k B_k^{\ast}L_1 B_k$,  then $A_j$ and $B_k$ can be assumed to be constant matrices. Let us denote $A_j, B_k$ as 
$$\left[
\begin{array}{cc}
P^{(j)}_1(x_1,x_2)&R^{(j)}_1(x_1,x_2)\\
P^{(j)}_2(x_1,x_2)&R^{(j)}_2(x_1,x_2)\\
\end{array}
\right]
\quad, 
\left[
\begin{array}{cc}
p^{(k)}_1(x_1,x_2)&r^{(k)}_1(x_1,x_2)\\
p^{(k)}_2(x_1,x_2)&r^{(k)}_2(x_1,x_2)\\
p^{(k)}_3(x_1,x_2)&r^{(k)}_3(x_1,x_2)\\
\end{array}
\right]
.$$ 

Comparing the entry $(11)$ in $\sum_j A_j^{\ast}A_j+\sum_k B_k^{\ast}L_1 B_k$ and $L_2$ gives
$$\sum_j\left((P_1^{(j)}(x_1,x_2))^2+(P_2^{(j)}(x_1,x_2))^2\right) +\sum_k\left((p_1^{(k)}(x_1,x_2))^2(1+x_1)+\right.$$
$$\left.+(p_2^{(k)}(x_1,x_2))^2(1+x_1+x_2)+(p_3^{(k)}(x_1,x_2))^2(1+x_2)\right)\overbrace{=}^{?}1+\frac{1}{3}x_1.$$
By observing the monomial of the form $Kx_i^{n}, n\in \NN, K\neq 0, i=1,2,$ of the highest degree on the left side, it can be seen, that monomials $Ax_i^{n}, n\in\NN, A\neq 0, i=1,2$, do not appear in any $P_i^{(j)}$ nor $p_i^{(k)}$. With the same reasoning applied to the entry $(22)$ not even in $R_i^{(j)}$ and $r_i^{(k)}$. Further on, since the monomial $Kx_1 ^mx_2^n$, $K\neq 0$, $m,n \in \NN$, from $P_i^{(j)}$ or $p_i^{(k)}$, does not multiply into monomial $K$ or $Kx_1$, $K\neq 0$, it is not needed to satisfy the upper equality. Similar reasoning can be applied to the other entries, hence WLOG $A_j$, $B_k$ are constant matrices.

The comparison of coefficients in $\sum_j A_j^{\ast}A_j+\sum_k B_k^{\ast}L_1 B_k$ and $L_2$ in a constant-matrix case gives the following equalities:

Entry $(11)$:
\begin{equation}\label{(1)}
x_2:\quad \sum_k \left((p^{(k)}_2)^2 + (p^{(k)}_3)^2 \right)= 0 \quad \Rightarrow \quad p^{(k)}_2=p^{(k)}_3= 0, \;  \forall k
\end{equation}
\begin{equation}\label{(2)}
x_1:\quad \sum_k \left((p^{(k)}_1)^2 + (p^{(k)}_2)^2\right) = \frac{1}{3}
\end{equation}
\begin{equation}\label{(3)}
1:\quad \sum_{j} \left(\sum_{i=1}^2(P^{(j)}_i)^2\right)+ \sum_k \left((p^{(k)}_1)^2 + (p^{(k)}_2)^2 + (p^{(k)}_3)^2\right) = 1
\end{equation}

Entry $(12)$(=entry $(21)$):
\begin{equation}\label{(4)}
x_2:\quad \sum_k \left(p^{(k)}_2 r^{(k)}_2 + p^{(k)}_3 r^{(k)}_3\right)= 0
\end{equation}
\begin{equation}\label{(5)}
x_1:\quad \sum_k \left(p^{(k)}_1 r^{(k)}_1 + p^{(k)}_2 r^{(k)}_2\right)= 0
\end{equation}
\begin{equation}\label{(6)}
1:\quad \sum_{j} \left(\sum_{i=1}^2P^{(j)}_i R^{(j)}_i\right)+ \sum_{k} \left(\sum_{i=1}^3 p^{(k)}_i r^{(k)}_i\right) = \frac{3}{4}
\end{equation}

Entry $(22)$:
\begin{equation}\label{(7)}
x_2:\quad \sum_k \left((r^{(k)}_2)^2 + (r^{(k)}_3)^2 \right)= \frac{1}{3}
\end{equation}
\begin{equation}\label{(8)}
x_1:\quad \sum_k \left((r^{(k)}_1)^2 + (r^{(k)}_2)^2\right) = 0 \quad \Rightarrow \quad r^{(k)}_1=r^{(k)}_2= 0, \;  \forall l
\end{equation}
\begin{equation}\label{(9)}
1:\quad \sum_{j} \left(\sum_{i=1}^2(R^{(j)}_i)^2\right)+ \sum_k \left((r^{(k)}_1)^2 + (r^{(k)}_2)^2 + (r^{(k)}_3)^2\right) = 1
\end{equation}

\hfill We will see, that the upper equalities cannot be simultaneously satisfied.\\
From \ref{(1)} and \ref{(8)} we conclude $\sum_{k} \left(\sum_{i=1}^3 p^{(k)}_i r^{(k)}_i\right)=0$. We use this in \ref{(6)} and get
$\sum_{j} \left(\sum_{j=1}^2 P^{(j)}_i R^{(j)}_i\right)= \frac{3}{4}$. Using $\ref{(1)}$ and \ref{(2)} in \ref{(3)} gives
$\sum_{j} \left(\sum_{i=1}^2(P^{(j)}_i)^2\right)= \frac{2}{3}$. Similarly using \ref{(7)} and \ref{(8)} in \ref{(9)} gives
$\sum_{j} \left(\sum_{i=1}^2(R^{(j)}_i)^2\right)= \frac{2}{3}$.

The following chain of (in)equalities should hold:
$$\frac{2}{3}=\frac{\frac{2}{3}+\frac{2}{3}}{2}=\sum_{j}\sum_{i=1}^2 \frac{(P^{(j)}_i)^2+ (R_i^{(j)})^2}{2} \geq  \sum_{j}\sum_{i=1}^2 \left|P^{(j)}_i R^{(j)}_i \right| \geq \sum_{j}\sum_{i=1}^2 P^{(j)}_i R^{(j)}_i = \frac{3}{4},$$
where the first inequality follows from AG-inequalites applied to pairs $\left\{(P^{(j)}_i)^2,\right.$ $\left. (R_i^{(j)})^2\right\}$, i.e., $\frac{(P^{(j)}_i)^2+ (R_i^{(j)})^2 }{2}\geq \left|P^{(j)}_i R^{(j)}_i \right|$, $\forall$ $i$, $j$.

We conclude $\frac{2}{3}\geq \frac{3}{4}$, which is obviously a contradiction.
\end{proof}

\section{Monicity of $L_1$}

In this section we show, that for diagonal $L_1$ monicity in Theorem $\ref{main}$ can be removed. In Proposition $\ref{smallest}$ we first prove the case of the set $D_{L_1}(1)$ being a singleton and then also the other cases of bounded $D_{L_1}(1)$ in Theorem $\ref{diagonal}$.

\begin{proposition} $\label{smallest}$
Suppose $L_1 \in \SSS\RR^{d\times d}[x]$ is a diagonal linear polynomial and the set $D_{L_1}(1)=\left\{a\right\}$. Then for every linear polynomial $L_2\in \SSS\RR^{\ell\times\ell}[x]$ with $L_2|_{D_{L_1}(1)} \succeq 0$, there are matrix polynomials $A_j\in \RR^{\ell\times\ell}\left[ x\right]$ and $B_k\in \RR^{d\times \ell}\left[x\right]$ satisfying 
$$L_2=\sum_jA_j^{\ast}A_j+\sum_kB_k^{\ast}L_1 B_k.$$
\end{proposition}

In the proof we will use the following proposition:

\begin{proposition}$\label{smallest2}$
For every $A\in\SSS\RR^{\ell\times \ell}$ there exist $B_k\in\SSS\RR^{2\times \ell}$, such that
$$\sum_k B_k^{\ast}
\left[ 
\begin{array}{cc}
x&0\\
0&-x
\end{array}
\right]
B_k = Ax.$$
\end{proposition}

\begin{proof}
According to the well-known fact every real symmetric matrix is real congruent to a diagonal $D$ with
elements $1$,$-1$ and $0$ on the diagonal, i.e., $A=\sum_k \tilde{B}_k^{\ast}D\tilde{B}_k$, where $D$, $\tilde{B}_k\in \SSS\RR^{\ell\times \ell}$. $Dx$ can be constructed from 
$\left[ 
\begin{array}{cc}
x&0\\
0&-x
\end{array}
\right]$ with the aim of equalities: 
\begin{eqnarray*}
E_{ii} &=& \left[\begin{array}{cc} e_i & 0 \end{array}\right]
\left[\begin{array}{cc} 1 & 0 \\ 0   &  -1 \end{array}\right]
\left[\begin{array}{c} e_i^{\ast}\\ 0 \end{array} \right],\\
-E_{ii} &=& \left[\begin{array}{cc} 0   &   e_i\\ \end{array} \right]
\left[\begin{array}{cc} 1   &   0\\ 0   &  -1 \end{array}\right]
\left[\begin{array}{c} 0\\ e_i^{\ast} \end{array}\right],
\end{eqnarray*}
where $e_i$ denotes the standard $\RR^{\ell\times 1}$ vector.
\end{proof}

\begin{proof} [Proof of Proposition \ref{smallest}]
Up to translation we may assume $D_{L_1}(1)=\left\{0\right\}$. For a polynomial

$$\tilde{L}_1=
\left[
\begin{array}{cc}
x_{1}&0\\
0&-x_{1}\\
\end{array}
\right]
\oplus
\cdots
\oplus
\left[
\begin{array}{cc}
x_{n}&0\\
0&-x_{n}\\
\end{array}
\right]$$ 
we have $D_{\tilde{L}_1}(1)=\left\{0\right\}$. After applying Theorem $\ref{farkas}$ to diagonal entries of $L_1$ and each diagonal entry of $\tilde{L}_1$, it follows $\tilde{L}_1=\sum_j \tilde{A}_j^{\ast}\tilde{A}_j +\sum_l \tilde{B}_k^{\ast}L_1 \tilde{B}_k$ for some constant $\tilde{A}_j$, $\tilde{B}_k$. Therefore it suffices to find $A_j, B_k$, such that $L_2=\sum_j A_j^{\ast}A_j +\sum_k B_k^{\ast}\tilde{L}_1 B_k$.

If we write $L_2(x)= R_0 + \sum_{i=1}^{n}R_i x_i$, then $L_2(0)=R_0\succeq 0$. So there exists $A$, such that $R_0=A^{\ast}A$. According to Proposition $\ref{smallest2}$, $R_i x_i$ can be expressed in a desired way with
$\left[
\begin{array}{cc}
x_i&0\\
0&-x_i
\end{array}
\right]
$, hence also with $\tilde{L}_1$.

Since $i$ was arbitrary, we are done.
\end{proof}

Now we will extend Proposition $\ref{smallest}$ to the general case. One additional lemma will be needed for that.

\begin{lemma}$\label{lemma1}$
Suppose $L=P_0+\sum_{i=1}^{n} P_i x_i\in \SSS\RR^{d\times d}[x]$ is a linear polynomial and let $0$ be an interior point of the set $D_{L_1}(1)$. Then there exists a monic linear polynomial $\tilde{L}=I+\sum_{i=1}^{n} \tilde{P}_i x_i\in \SSS\RR^{\tilde{d}\times \tilde{d}}[x]$, where $\tilde{d}\leq d$ and $D_{L}(1)=D_{\tilde{L}}(1)$, such that $\tilde{L}=C^{\ast} L C$ and $L=D^{\ast} \tilde{L} D$, where $C\in \RR^{d\times \tilde{d}}$, $D\in \RR^{\tilde{d}\times d}$.
\end{lemma}

\begin{proof}
Since $0$ is an interior point, $P_0\succeq 0$ and $\Im(P_i)\subseteq \Im(P_0)$ for $i=1,\ldots,n$ (See [$2$, Proof of Proposition $2.1$].). We have $P_0=V^{\ast}DV$, where $D$ is diagonal and $V$ orthogonal. Further on $V^{\ast}LV=L|_{\Im(P_0)}\oplus 0_{d-\tilde{d}}$, $\tilde{d}=\dim(\Im(P_0))$. Hence $L|_{\Im(P_0)}=J^{\ast}(V^{\ast}LV)J$ with $J^{\ast}:=[I_{\tilde{d}}\quad 0^{\tilde{d}\times (d-\tilde{d})}]\in \RR^{\tilde{d}\times d}$.  Defining $\tilde{P}_{0}:=P_0|_{\Im(P_0)}\succ 0$, gives $\tilde{P}_{0}=B^{\ast}B$, where $\tilde{P}_{0},B\in \RR^{\tilde{d}\times\tilde{d}}$ and $B$ is invertible. So $(B^{-1})^{\ast}L|_{\Im(P_0)}B^{-1}=(B^{-1})^{\ast}B^{\ast}BB^{-1}+ \sum_{i}(B^{-1})^{\ast}P_i|_{\Im(P_0)}B^{-1}x_i=I+\sum_{i}(B^{-1})^{\ast}P_i|_{\Im(P_0)}B^{-1}x_i=:\tilde{L}$. $\tilde{L}$ is in $\RR^{\tilde{d}\times \tilde{d}}$ and $D_{L}(1)=D_{\tilde{L}}(1)$.
With $C^{\ast}:=(B^{-1})^{\ast}J^{\ast}V^{\ast}\in \RR^{\tilde{d}\times \tilde{d}}\RR^{\tilde{d}\times d}\RR^{d\times d}=\RR^{\tilde{d}\times d}$ and $D^{\ast}:=(V^{-1})^{\ast}JB^{\ast}\in \RR^{d\times d}\RR^{d\times \tilde{d}}\RR^{\tilde{d}\times \tilde{d}}=\RR^{d\times \tilde{d}}$, the lemma is proved.
\end{proof}

\begin{theorem} $\label{diagonal}$
Suppose $L_1 \in \SSS\RR^{d\times d}[x]$ is a diagonal linear polynomial and the set $D_{L_1}(1)$ is bounded. Then for every linear polynomial $L_2\in \SSS\RR^{\ell\times \ell}[x]$ with $L_2|_{D_{L_1}(1)} \succ 0$, there are matrix polynomials $A_j\in \RR^{\ell\times \ell}\left[ x\right]$ and $B_k\in \RR^{d\times \ell}\left[x\right]$ satisfying 
$$L_2=\sum_jA_j^{\ast}A_j+\sum_kB_k^{\ast}L_1 B_k.$$
\end{theorem}

\begin{proof}
If $D_{L_1}(1)=\emptyset$, then by Theorem \ref{farkas}, $-1$ is in the convex cone generated by the diagonal entries of $L_1$. Now $L_2$ can be expressed in the desired form, since the quadratic module generated by $-1$ in any ring with involution consists of all symmetric elements by the identity $4a=(a+1)^2-(a-1)^2$.
If $D_{L_1}(1)=\left\{\vec{a}\right\}$, then we can use Proposition \ref{smallest} and we are done. If $\dim D_{L_1}(1)=n$, then by Lemma $1$ WLOG $L_1$ is monic and Theorem \ref{main} is used. Otherwise we have $1\leq \dim D_{L_1}(1)=:k \leq n-1$. Since $D_{L_1}(1)$ is convex, it lies in some affine subspace of dimension $k$. With translation WLOG $\vec{0}\in D_{L_1}(1)$ and hence the affine subspace is actually a vector subspace of dimension $k$. Let $B=\left\{e_1^{\prime},e_2^{\prime},\ldots,e_k^{\prime}\right\}$ be the basis of this subspace. $B$ can be completed to the basis of $\RR^{n}$, i.e., $B^{\prime}=\left\{e_1^{\prime}, \ldots, e_k^{\prime}, e_{k+1}^{\prime}, \ldots, e_{n}^{\prime}\right\}$. Standard basis $\left\{e_1, e_2, \ldots, e_{n}\right\}$ of $\RR^{n}$ can be uniquelly expressed by $B^{\prime}$ and vice versa, i.e., $e_i=\sum_{j=1}^{n}\alpha^{(i)}_je_j^{\prime}$ and $e_i^{\prime}=\sum_{j=1}^{n}\beta^{(i)}_je_j$, for unique $\alpha_j^{(i)}, \beta_j^{(i)}\in\RR$. Therefore introducing new unknows $x_i^{\prime}$ as $x_i^{\prime}=\sum_{j=1}^{n}\beta^{(i)}_jx_j$ gives also $x_i=\sum_{j=1}^{n}\alpha^{(i)}_jx_j^{\prime}$. Putting expressed $x_i$-s into $L_1(x_1,\ldots,x_n)$, we get $\tilde{L}_1(x_1^{\prime},\ldots,x_n^{\prime})$. The map $\Phi:\RR^n\rightarrow \RR^n$, defined by $\Phi:(a_1,\ldots,a_n)\mapsto (\sum_{j=1}^{n}\beta^{(1)}_ja_j,\ldots,\sum_{j=1}^{n}\beta^{(n)}_ja_j)$, is bijective and  $L_1((a_1,\ldots,a_n))=\tilde{L}_1(\Phi(a_1,\ldots,a_n))$.
Hence $\Phi(D_{L_1}(1))=D_{\tilde{L}_1}(1)$. So $D_{L_1}(1)$ and $D_{\tilde{L}_1}(1)$ are in bijective correspondence. Similarly for $D_{L_2}(1)$ and $D_{\tilde{L}_2}(1)$. Therefore $D_{L_1}(1)\subseteq D_{L_2}(1) \Leftrightarrow D_{\tilde{L}_1}(1) \subseteq D_{\tilde{L}_2}(1)$. From the construction of basis $B^{\prime}$, $x^{\prime}\in D_{\tilde{L}_1}(1)$ is of the form $(x_1^{\prime},\ldots,x_k^{\prime}, 0,\ldots,0)$.

Let us write $$\tilde{L}_1=\left(P^{\prime}_0+\sum_{i=1}^{k}P^{\prime}_i x^{\prime}_i\right) + \sum_{i=k+1}^{n}P^{\prime}_{i} x^{\prime}_{i}= \tilde{L}_{1,1}(x_1^{\prime},\ldots,x_k^{\prime})+\tilde{L}_{1,2}(x_{k+1}^{\prime},\ldots,x_n^{\prime}).$$ 
We notice, that $\tilde{L}_1$ is still diagonal. For
$$\tilde{L}=\tilde{L}_{1,1}(x_1^{\prime},\ldots,x_k^{\prime})\oplus 
\left[
\begin{array}{cc}
x_{k+1}^{\prime}&0\\
0&-x_{k+1}^{\prime}\\
\end{array}
\right]
\oplus
\cdots
\oplus
\left[
\begin{array}{cc}
x_{n}^{\prime}&0\\
0&-x_{n}^{\prime}\\
\end{array}
\right]
$$ (which is obviously diagonal),
$D_{\tilde{L}_1}(1)=D_{\tilde{L}}(1)$, and with the use of Theorem $\ref{farkas}$ on diagonal entries of $\tilde{L}_1$ and each diagonal entry of $\tilde{L}$, $\tilde{L}$ can be expressed as $\sum_j A^{\ast}_jA_j+ \sum_k B^{\ast}_k \tilde{L}_1 B_k$. Hence it suffices to prove the statement of the theorem for the pair $\tilde{L}, \tilde{L}_2$.

Analogously as for $\tilde{L}_1$ we write $\tilde{L}_2$ as $\tilde{L}_2=\left(R^{\prime}_0+\sum_{i=1}^{k}R^{\prime}_i x^{\prime}_i\right) + \sum_{i=k+1}^{n}R^{\prime}_{i} x^{\prime}_{i}= \tilde{L}_{2,1}(x_1^{\prime},\ldots,x_k^{\prime})+\tilde{L}_{2,2}(x_{k+1}^{\prime},\ldots,x_n^{\prime})$. 
We have $\tilde{L}_{2,1}|_{D_{\tilde{L}_{1,1}}(1)}\succ 0$. Since there exists an interior point in $D_{\tilde{L}_{1,1}}(1)$, Lemma $\ref{lemma1}$ allows us to regard $\tilde{L}_{1,1}$ as monic. Finally Theorem $\ref{main}$ is used for the pair $\tilde{L}_{1,1},\tilde{L}_{2,1}$.

It remains to express $\tilde{L}_{2,2}(x_{k+1}^{\prime},\ldots,x_n^{\prime})=R^{\prime}_{k+1} x^{\prime}_{k+1}+\cdots+R^{\prime}_n x^{\prime}_n$ with $\tilde{L}.$
According to Proposition $\ref{smallest2}$, $R_i x^{\prime}_i$ can be expressed with
$\left[ 
\begin{array}{cc}
x^{\prime}_i&0\\
0&-x^{\prime}_i
\end{array}
\right].$
Hence also with $\tilde{L}$. Since $i$ was arbitrary, we are done.

To conclude, we got the expression 
$$\tilde{L}_2(x_1^{\prime},\ldots,x_{n}^{\prime})=\sum_j \tilde{A}_j^{\ast}\tilde{A}_j+\sum_k \tilde{B}_k^{\ast}\tilde{L}_1(x_1^{\prime},\ldots,x_{n}^{\prime}) \tilde{B}_k.$$
Using $x_i^{\prime}=\sum_{j=1}^{n}\beta^{(i)}_jx_j$, we finally get
$$L_2(x_1,\ldots,x_n)=\sum_j A_j^{\ast}A_j+\sum_k B_k^{\ast}L_1(x_1,\ldots,x_n) B_k.$$
\end{proof}

\section{Strict positivity of $L_2|_{D_L(1)}$}

The next thing to be studied is the necessity of positive definiteness in Theorem \ref{main}, i.e., whether semidefiniteness suffices.  We separately study the one-variable case from the general diagonal case.

\subsection{One-variable case}

\begin{theorem}$\label{one-variable}$
Suppose $L_1(x)=P_0+P_1 x\in \SSS\RR^{d\times d}[x]$ is a linear polynomial and the set $D_{L_1}(1)$ has an interior point. Then
for every linear polynomial $L_2(x)=R_0+R_1 x\in \SSS\RR^{\ell\times \ell}[x]$ with $L_2|_{D_{L_1}(1)} \succeq 0$, there are matrix polynomials $A_j\in \RR^{\ell\times \ell}\left[ x\right]$ and $B_k\in \RR^{d\times \ell}\left[x\right]$ satisfying 
$$L_2=\sum_jA_j^{\ast}A_j+\sum_kB_k^{\ast}L_1 B_k.$$
\end{theorem}

\begin{proof}
\textbf{Case 1:} If $D_{L_1}(1)$ is bounded and there exists an interior point in $D_{L_1}(1)$, then according to Lemma $\ref{lemma1}$, we can assume $L_1$ and $L_2$ are monic, i.e., $P_0=I, R_0=I$. Since we have just one variable, we may interpret both $L_1$ and $L_2$ as NC polynomials, or precisely linear pencils. We will first show, that\\
$\underline{D_{L_1}(1)\subseteq D_{L_2}(1)\Rightarrow D_{L_1}\subseteq D_{L_2}:}$ Let us take $X\in D_{L_1}$, $X\in\SSS\RR^{m\times m}$, which means $L_1(X)=I\otimes I_m+P_1\otimes X\succeq 0$. Or equivallently $I_m\otimes I+X\otimes P_1\succeq 0$. We have to show that $X\in D_{L_2}$. Since $X$ is symmetric, it can be real ortogonally diagonalized, i.e., $UXU^{T}=D$, where $U$ is an orthogonal matrix of size $m$. After multiplying with invertible matrix
$U\otimes I$ we get $(U\otimes I)(I_m\otimes I+X\otimes P_1)(U\otimes I)^T=I_m\otimes I+D \otimes P_1 \succeq 0$. Hence $X\in D_{L_1}\Leftrightarrow D\in D_{L_1}$. Now $I_m\otimes I+D\otimes P_1\succeq 0$ is a block-diagonal matrix with the blocks of the form $I+d_i P_1$. It follows $I_m\otimes I+D\otimes P_1 \succeq 0 \Leftrightarrow I+d_i P_1\succeq 0$ for $i=1,2,\ldots,m$ $\Leftrightarrow d_i\in D_{L_1}(1)$ for $i=1,2,\ldots,m$. But according to the assumption $d_i\in D_{L_2}(1)$ for all $i=1,2,\ldots,m$ and hence $X \in D_{L_2}$.

To be able to use LP-satz (i.e., Theorem \ref{main} with $L_2|_{D_{L_1}(1)}\succ 0$ replaced by $L_2|_{D_{L_1}}\succeq 0$ and $A_j, B_k$ constant); cf. [$2$, Corollary $3.7$]; for the pair $L_1, L_2$, $D_{L_1}$ must be bounded. But by [$2$, Proposition $2.4$] this is equivalent to $D_{L_1}(1)$ being bounded. So by LP-satz there exist $B_k$, such that $L_2=\sum_k B_k^{\ast}L_1 B_k$.\\
\textbf{Case 2:}
If $D_{L_1}(1)$ is unbounded, then it is an interval of the form $[a,\infty)$, $(-\infty,a]$, $(-\infty,\infty)$, $a\in \RR$. With translation we may assume $a=0$.

First we study the case $D_{L_1}(1)=[0,\infty)$. Since $0\in D_{L_1}(1)$, we have $P_0\succeq 0$. We can also show, that $P_1\succeq 0$. To explain: $u^{\ast}L_1(x)u= u^{\ast}P_0 u+ u^{\ast}(P_1x)u = u^{\ast}P_0 u+x u^{\ast}P_1 u$. In the case that $u^{\ast}P_1 u\neq 0$, we have $x\left|u^{\ast}P_1 u\right|>\left|u^{\ast}P_0 u\right|$ for $x$ great enough. Therefore, if there exists $u$, such that $u^{\ast}P_1u<0$, then $\displaystyle\lim_{x\to \infty}x\notin D_{L_1}(1)$. Contradiction.

Since $P_0$ and $P_1$ are positive semidefinite, we can use Newcomb's theorem [$5$, Theorem $20.2.2$] (It is actually made for complex matrices but with a slight modification of the proof it holds for real as well.) to simultaneously diagonalize them with invertible $S$, i.e., $S^{\ast}P_0S$, $S^{\ast}P_1S$ are both diagonal. So WLOG $L_1$ is diagonal. Analogously for $L_2$. Now we just use Theorem $1$ on diagonal entries of $L_1$ and each diagonal entry of $L_2$ and we are done.

In the case $D_{L_1}(1)=(-\infty,0]$, we have again $P_0\succeq 0$. As above we show $P_1\preceq 0$.
Since $P_0,$ $P_1$ are semidefinite, Newcomb's theorem [$5$, Theorem $20.2.2$] can be used and we proceed as above.
 
In the case $D_{L_1}(1)=(-\infty,\infty)$, we have $P_0 \succeq 0$ and it is easy to show, that $P_1=0$. Therefore $L_1(x)=P_0$ and analogously $L_2(x)=R_0$, where $R_0\succeq 0$. Hence $L_2(x)=C^{\ast}C$.
\end{proof}

The following example shows, that $D_{L_1}(1)$ must have an interior point in Theorem $\ref{one-variable}$.

\begin{example}
For the non-monic, non-diagonal polynomial 
$ L_1(x)=\tiny {[\begin{array}{cc} 1  &   x\\ x  &   0 \end{array}]},$
we have that the set $D_{L_{1}}(1)=\left\{0\right\}$. Therefore, $D_{L_1}(1)$ is non-empty and bounded. It also holds that the polynomial $L_2(x)=x$ is non-negative on $D_{L_{1}}(1)$, but there do not exist matrix polynomials $A_j\in \RR[x], B_k\in \RR^{2\times 1}[x]$, such that $L_2=\sum_j A_j^{\ast}A_j+\sum_k B_k^{\ast}L_1 B_k$.
\end{example}

\begin{proof}
Since $\det(L_1)=-x^2$, $D_{L_{1}}(1)=\left\{0\right\}$. It is obvious, that $L_2|_{D_{L_{1}}(1)}\geq 0$. The proof will be by contradiction. Let us say there exist $A_j$, $B_k$, such that $\sum_j A_j^{\ast}A_j+\sum_k B_k^{\ast}L_1 B_k=x$.
Let $B_k$ be of the form $[b^{(k)}_1,b^{(k)}_2]^T$, where $b^{(k)}_i\in \RR[x]$ and $A_j\in \RR[x]$. Comparing the expression $\sum_j A_j^{\ast}A_j+\sum_k B_k^{\ast}L_1 B_k$ with $x$:
$$\sum_j A_j^2 + \sum_k \left((b^{(k)}_1)^2+2(b_1^{(k)})(b_2^{(k)})x\right)\overbrace{=}^{?}x.$$
The coefficient at $1$ on LHS equals $\sum_j A_{j,0}^2+\sum_k (b^{(k)}_{1,0})^2$, where $A_{j,0}$ denotes the free monomial in $A_j$ and $b^{(k)}_{1,0}$ the free monomial in $b^{(k)}_1$. Since on RHS it is $0$, $A_{j,0}=b^{(k)}_{1,0}=0$ for all $j,k.$ But then the coefficient at $x$ on LHS is $0$, while on RHS $1.$ Contradiction.
\end{proof}

\subsection{General diagonal case}

The aim of this subsection is to prove that for very special diagonal $L_1$, we can replace the condition $L_2|_{D_{L_1}(1)} \succ 0$ in Theorem \ref{main} with the weaker condition $L_2|_{D_{L_1}(1)} \succeq 0$. More precisely, we have the following theorem.

\begin{theorem}$\label{simplex}$
Suppose the polynomial $L_1=P_0+\sum_{i=1}^nP_i x_i \in \SSS\RR^{d\times d}[x]$ is diagonal and the set $D_{L_1}(1)$ is an $n$-simplex. If the polynomial $L_2=R_0+\sum_{i=1}^nR_i x_i \in \SSS\RR^{\ell\times \ell}[x]$, satistfies the condition $L_2|_{D_{L_1}(1)}\succeq 0$, then there exist matrix polynomials $A_j\in \RR^{\ell\times \ell}, B_k\in \RR^{d\times \ell}$, such that the following is true:
$$L_2=\sum_j A^{\ast}_jA_j+ \sum_k B^{\ast}_k L_1 B_k.$$
\end{theorem}

In the proof, we will need the following well known result.

\begin{theorem}[1, Theorem 4] \label{extension}
Let $A$ be commutative $C^{\ast}$-algebra and $\tau:A\rightarrow S$ a positive linear function, where $S$ is a vector subspace of the algebra of all bounded operators on some Hilbert space. Then $\tau$ is completely positive.
\end{theorem}

\begin{proof}[Proof of Theorem \ref{simplex}]
An $n$-simplex in $\RR^n$ is an intersection of $n+1$ halfspaces. Therefore, it can be defined as $D_{L}(1)$ of $L=\bigoplus_{i=1}^{n+1}\left(a_0^{(i)}+\sum_{j=1}^{n}a_{j}^{(i)}x_j\right)=\tilde{P}_0+\sum_{i=1}^n\tilde{P}_i x_i \in \SSS\RR^{(n+1)\times (n+1)}[x]$, for appropriate $a_j^{(i)}\in\RR$. By Theorem $1$,
$L=\sum_k A_k^{\ast}L_1A_k$. Hence, it suffices to prove the statement for the pair $L, L_2$.

There is an interior point $v=(v_1,v_2,\ldots,v_n)\in\RR^{n}$ in the $n$-simplex $D_{L_1}(1)$. With substitutions $x_i=\tilde{x}_i+v_i$, the interior point $v$ of $D_{L_1}(1)$ becomes the interior point $0$ for $\tilde{L}=\tilde{P}_0+\sum_{i=1}^n \tilde{P}_i \tilde{x}_1$. $\tilde{L}$ is also diagonal, $\tilde{P}_0=L_1(v)$ and $D_{\tilde{L}}(1)=D_{L}(1)-v$. The same is true for $L_2$ and $\tilde{L}_2$. Since $D_{L}(1)\subseteq D_{L_2}(1) \Leftrightarrow D_{L}(1)-v\subseteq D_{L_2}(1)-v \Leftrightarrow D_{\tilde{L}}(1)\subseteq D_{\tilde{L}_2}(1),$ we have $\tilde{L}_2|_{D_{\tilde{L}}(1)}\succeq 0$. Since $\tilde{P}_0=L_1(v)$, $\tilde{P}_0$ is invertible and we have
$\left((P_{0})^{-\frac{1}{2}}\right)^{\ast}\tilde{L}_1(P_{0})^{-\frac{1}{2}}=$$
\left((P_{0})^{-\frac{1}{2}}\right)^{\ast}\tilde{P}_0(P_{0})^{-\frac{1}{2}}+
\sum_i \left((P_{0})^{-\frac{1}{2}}\right)^{\ast}(\tilde{P}_i \tilde{x}_i)(P_{0})^{-\frac{1}{2}}=$
$I+\sum_i \hat{P}_i \tilde{x}_i=:\hat{L}(\tilde{x})$,
where $D_{\tilde{L}}(1)= D_{\hat{L}}(1)$. Since $0$ is an interior point in $D_{\tilde{L}_2}(1)$, by Lemma $\ref{lemma1}$ there exists monic $\hat{L}_2$, such that $D_{\tilde{L}_2}(1)=D_{\hat{L}_2}(1)$ and $\hat{L}_2=C^{\ast}\tilde{L}_2C$, $\tilde{L}_2=D^{\ast}\hat{L}_2 D$. Therefore, it suffices to prove the statement for $\hat{L}_1,\hat{L}_2$.

Now we define vector spaces $\hat{S}_1:=\Lin\left\{I,\hat{P}_1,\ldots,\hat{P}_n\right\}$ and $\hat{S}_2:=\Lin\left\{I,\hat{R}_1,\right.$ $\left. \ldots,\hat{R}_n\right\}$. Since $D_{\hat{L}}(1)$ is bounded, $\left\{I,\hat{P}_1,\ldots, \hat{P}_n\right\}$ is lineary independent by [$2$, Proposition $2.6$] in $\DD\RR^{(n+1)^2}=\left\{\mathrm{diagonal}\;\mathrm{matrices}\;\mathrm{of}\;\mathrm{size}\;n+1\right\}$. Hence, it is also its basis, which implies $\hat{S}_1$ is algebra. Therefore, $\tau: \hat{S}_1 \rightarrow \hat{S}_2$, where $I\mapsto I$ and $\hat{P}_i\mapsto \hat{R}_i$, is a well-defined unital linear map. 
By [$2$, Theorem $3.5$], it is also positive.\\
Since all matrices in $\left\{I,\hat{P}_1,\ldots,\hat{P}_n\right\}$ are diagonal, $\hat{S}_1$ is commutative algebra. Let $\hat{S}^{\CC}_1$ be complex linear span of $\left\{I,\hat{P}_1,\ldots,\hat{P}_n\right\}$. Similarly for $\hat{S}_2, \hat{S}^{\CC}_2$. Now we extend $\tau$ to $\tau^{\CC}:\hat{S}^{\CC}_1\rightarrow \hat{S}^{\CC}_2$, where $\tau^{\CC}(I/\hat{P}_i)=I/\hat{R}_i$. Since positive elements from $\hat{S}^{\CC}_1$ are in $\hat{S}_{1}$ and $\tau^{\CC}|_{\hat{S}_1}=\tau$, $\tau^{\CC}$ is positive. Taking $\hat{S}^{\CC}_1$ as $A$  and $\hat{S}^{\CC}_2$ as $S$ in Theorem $\ref{extension}$, $\tau^{\CC}$ is in fact completely positive. Also, $\tau^{\CC}|_{\hat{S}_1}=\tau$ is completely positive. By [$2$, Theorem $3.5$] $D_{L_1}\subseteq D_{L_2}$. By LP-satz [$2$, Corollary $3.7$] for the pair $\hat{L}_1, \hat{L}_2$, there exist $V_j\in \RR^{d\times m}$ and $\mu\in\NN$, such that $ \hat{L}_2=\sum_{j=1}^{\mu}V_{j}^{\ast}\hat{L}_1 V_j$. The theorem is proved.
\end{proof}

\noindent \textbf{Remark.} The matrix polynomials $A_j, B_k$ in Theorem $\ref{simplex}$ are constant while in Theorem \ref{main} they are matrix polynomials.\\

\noindent \textbf{Remark.} If we replace the expression \textit{the set $D_{L_1}(1)$ is an $n$-simplex }in Theorem \ref{simplex} with the expression \textit{the set $\Lin\left\{P_0, P_1,\ldots,P_n\right\}$ is an algebra and has an interior point}, the theorem still holds. Indeed, since the set $\Lin\left\{P_0, P_1,\ldots,P_n\right\}$ is algebra, it is isomorphic to a subalgebra $A$ in $\RR^{n+1}$. Further on, we may assume that this subalgebra separates the $n+1$ components of the vector (i.e., given $1\leq i<j\leq n+1$, we find an element of this algebra, that has distinct $i$-th and $j$-th component). This assumption does not harm, since otherwise there is redundancy in $L_1$ with respect to the polyhedron defined by $L_1$. By identifying $\RR^{n+1}$ with the set of all continuous functions from the set $X=\left\{1,2,\ldots, n\right\}$ to $\RR$, i.e., with $C(X)$, and using Stone Weierstrass theorem, we conclude that the algebra $A$ is dense in $\RR^{n+1}$. But because of finite dimensionality, it is then equal to the full algebra $\RR^{n+1}$. Therefore, the set $\left\{P_0, P_1,\ldots,P_n\right\}$ is lineary independent and since the polyhedron defined by its elements has an interior point, it is an $n$-simplex.\\

\noindent \textbf{Acknowledgments.} I would like to thank my mentor Prof. Dr. Jaka Cimpri\v c for proposing a problem, carefully reading the developing material, giving many valuable comments with ideas for further work and finally helping on correcting mistakes in preliminary versions of this article.

\end{document}